\documentclass[12pt]{amsart}
\usepackage{amssymb}
\usepackage{color}
\usepackage{mathrsfs, amssymb, array}
\usepackage[normalem]{ulem}
\usepackage{rotating}
\usepackage{extarrows}

\usepackage{bookmark}
\usepackage{verbatim}
\usepackage{mathcomp}
\usepackage{upgreek}
\usepackage[matrix,arrow,curve]{xy}
\usepackage{longtable}
\usepackage{adjustbox}
\usepackage[shortcuts]{extdash}
\makeatletter
\@addtoreset{equation}{subsection}
\makeatother

\newcommand{\gon}{\mathrm{gon}}
\newcommand{\covgon}{\mathrm{cov.gon}}

\newcommand{\C}{\mathbb{C}}

\newcommand{\p}{\mathbb{P}}
\newcommand{\QQ}{\mathbb{Q}}

\newcommand{\Cl}{\operatorname{Cl}}
\newcommand{\Bir}{\operatorname{Bir}}

\newcommand{\Aut}{\operatorname{Aut}}

\newcommand{\NE}{\operatorname{NE}}

\newcommand{\Q}{\mathbin{\sim_{\scriptscriptstyle{\Q}}}}

\renewcommand{\to}{ \, \tikz[baseline=-.6ex] \draw[->,line width=.5] (0,0) -- +(.5,0); \, }
\newcommand{\rat}{ \, \tikz[baseline=-.6ex] \draw[->,densely dashed,line width=.5] (0,0) -- +(.5,0); \, }
\newcommand{\ps}{ \, \tikz[baseline=-.6ex] \draw[->,dotted,line width=.6] (0,0) -- +(.5,0); \,}

\newcommand\iso{\stackrel{\simeq}{\to}}

\renewcommand{\thesubsection}{\arabic{section}.\arabic{subsection}}

\swapnumbers
\theoremstyle{plain}
\newtheorem{theorem}[subsection]{Theorem}
\newtheorem{lemma}[subsection]{Lemma}
\newtheorem{proposition}[subsection]{Proposition}

\newtheorem{construction}[subsection]{Construction}

\newtheorem{corollary}[subsection]{Corollary}

\newtheorem*{claim*}{Claim}

\theoremstyle{definition}
\newtheorem{definition}[subsection]{Definition}
\newtheorem*{definition*}{Definition}

\newtheorem*{notation*}{Notation}

\newtheorem{question}[subsection]{Question}

\newtheorem{remark}[subsection]{Remark}

\newcounter{NN}

\newcounter{NO}

\newcommand{\FT}{$\mathrm{FT}^{\mathrm{t}}$ }

\usepackage{tikz}
\usetikzlibrary{cd,arrows}
\tikzset{>=stealth}
\tikzcdset{arrow style=tikz}
\tikzset{link/.style={column sep=1.8cm,row sep=0.16cm}}

\author{J\'er\'emy Blanc}
\address{J\'er\'emy Blanc, Departement Mathematik und Informatik,  Universit\"at Basel, Spiegelgasse 1, 4051 Basel, Switzerland}
\email{jeremy.blanc@unibas.ch}
\author{Ivan Cheltsov}
\address{Ivan Cheltsov, School of Mathematics, The University of Edinburgh,  Edinburgh EH9 3JZ, UK.}
\email{I.Cheltsov@ed.ac.uk}
\author{Alexander Duncan}
\address{Alexander Duncan, Department of Mathematics, University of South Carolina,
Columbia, SC 29208, USA.}
\email{duncan@math.sc.edu}
\author{Yuri Prokhorov}
\address{Yuri Prokhorov, Steklov Mathematical Institute of Russian Academy of Sciences, Moscow, Russia}
\email{prokhoro@mi-ras.ru}

\title{Birational self-maps of threefolds of (un)-bounded genus or gonality}

\begin{document}

\begin{abstract}
We study the complexity of birational self-maps of a projective threefold $X$ by looking at the birational type of surfaces contracted. These surfaces are birational to the product of the projective line with a  smooth projective curve. We prove that the genus of the curves occuring is unbounded if and only if $X$ is birational to a conic bundle or a fibration into cubic surfaces. Similarly, we prove that the gonality of the curves is unbounded if and only if $X$ is birational to a conic bundle.
\end{abstract}
\maketitle
\tableofcontents
\section{Introduction}
Let $X$ be a smooth  projective complex algebraic variety. One way of studying the complexity of the geometry of elements of the group $\Bir(X)$ of birational self-maps of $X$ consists of studying the complexity of the irreducible hypersurfaces contracted by elements of $\Bir(X)$ (remember that an irreducible hypersurface $H$ is contracted by $\varphi\in \Bir(X)$ if $\varphi(H)$ has codimension $\ge 2$ in $X$). If $X$ is a curve, then $\Bir(X)=\Aut(X)$, so there is nothing to be said. If $X$ is a surface, every irreducible curve contracted by an element of $\Bir(X)$ is rational. The case of threefolds is then the first interesting to study in this context.

If $\dim(X)=3$, then every irreducible surface contracted by a birational transformation  $\varphi\in\Bir(X)$ is birational to $\p^1\times C$ for some smooth projective curve~$C$. There are then two natural integers that one can associate to $C$ in this case, namely its genus $g(C)$ and its gonality $\gon(C)$ (the minimal degree of a dominant morphism $C\to \p^1$). We then define the \emph{genus} $g(\varphi)$ (respectively the \emph{gonality} $\gon(\varphi)$) of $\varphi$ to be the maximum of the genera $g(C)$ (respectively of the gonalities $\gon(C)$) of the smooth projective curves $C$ such that a hypersurface of $X$ contracted by $\varphi$ is birational to $\p^1\times C$.

This notion of genus of elements of $\Bir(X)$ was already defined in \cite{Frumkin} with another definition, which is in fact equivalent to the definition above by \cite[Proposition~3]{Lamy2014}.  Moreover, for each $g$, the set of elements of $\Bir(X)$ of genus $\le g$ form a subgroup, so we get a natural filtration on $\Bir(X)$, studied in \cite{Frumkin,Lamy2014}. This naturally raises the question of finding the threefolds $X$ for which this filtration is infinite, namely the threefolds $X$ for which the genus of $\Bir(X)$ is unbounded (see \cite[Question~11]{Lamy2014}). Analogously, we get a filtration given by the gonality. Of course, the gonality is bounded if the genus is bounded, the unboundedness of the gonality is stronger than the unboundedness of the genus.

Note that the boundedness of the genus (respectively of the gonality) of elements of $\Bir(X)$ is a birational invariant. Our main result (Theorem~\ref{MainTheorem}) describes the threefolds having this property.

Recall that a variety $Y$ is a \emph{conic bundle} (respectively a \emph{del Pezzo fibration of degree $d$}) if $Y$ admits a morphism $Y\to S$ such that the generic fibre is a conic (respectively a del Pezzo surface of degree $d$) over the field of rational functions of $S$. If the conic has a rational point (or equivalently the conic bundle has a rational section), then it is isomorphic to $\p^1$ and in this case we say that the conic bundle is \emph{trivial}.

\begin{theorem}\label{MainTheorem}
Let $X$ be a smooth projective complex algebraic threefold.

\begin{enumerate}
\item\label{ConicBundleDim3Unbounded}
If $X$ is birational to a conic bundle, the gonality and the genus of the elements $\Bir(X)$ are both unbounded.
\item\label{DelPezzo3GenusUnbounded}
If $X$ is birational to a del Pezzo fibration of degree $3$, the genus of the elements of $\Bir(X)$ is unbounded.
\item\label{NotConicBundleGonalityBounded}
If $X$ is not birational to a conic bundle, the gonality of the elements of $\Bir(X)$ is bounded.
\item\label{NotConicBundleNotDp3GenusBounded}
If $X$ is not birational to a conic bundle and to a del Pezzo fibration of degree $3$, then both the genus and the gonality of elements of $\Bir(X)$ are bounded.
\end{enumerate}
\end{theorem}

We can generalise the above notions to higher dimensions. If $X$ is a smooth projective variety of dimension $d\ge 3$, every irreducible hypersurface contracted by an element of $\Bir(X)$ is birational to $\p^1\times S$ for some variety $S$ of dimension $d-2$. When $d\ge 4$, $\dim(S)\ge 2$ and there are then many ways to study the complexity of this variety. One possibility is the covering gonality $\covgon(S)$ of $S$, namely the smallest integer $c$ such that through a general point of $S$ there is an irreducible curve $\Gamma\subseteq S$ birational to a smooth curve of gonality $\le c$. As before, we say that \emph{the covering gonality of $\Bir(X)$ is bounded} if the covering gonality of the irreducible varieties $S$ such that a hypersurface contracted by an element $\Bir(X)$ is birational to $S\times \p^1$ is bounded. Since the covering gonality of a smooth curve is its gonality, this notion is the same as the gonality defined above, in the case of threefolds. As in dimension $3$, this is again a birational invariant.

In Corollary~\ref{corollary:solid-any-dimensions}, we prove  that if $X$ is a solid Fano variety (see \cite[Definition~1.4]{AhmadinezhadOkada}),
then the covering gonality of elements of $\Bir(X)$ are bounded by a constant that depends only on $\mathrm{dim}(X)$.
In particular, the covering gonality of birational selfmaps of birationally rigid Fano varieties of dimension $n$ (see \cite{CheltsovUMN}, \cite{PukhlikovAMS}
or \cite[Definition~3.1.1]{CheltsovShramov2015CRC}) are bounded by a constant that depends only on $n$.

In Proposition~\ref{Prop:ConicBundlesUnbounded}, we prove that if $\pi\colon X\to B$ is a trivial conic bundle of any dimension $\ge 3$, then the covering
gonality of the elements of
\[
\Bir(X/B)=\{\varphi\in \Bir(X)\mid \pi \circ
\varphi=\pi\}\subseteq \Bir(X)
\]
is unbounded. This raises the following two
questions:
\begin{question}
Let $B$ be a projective variety of dimension $\ge 3$ and let $X\to B$ be a non-trivial conic bundle. Is the covering gonality of elements of $\Bir(X/B)$ unbounded?
\end{question}
\begin{question}
Let $X$ be a projective variety of dimension $\ge 4$ that is not birational to a conic bundle. Is the covering gonality of elements of $\Bir(X)$ bounded?
\end{question}

A rough idea of the proof of Theorem~\ref{MainTheorem} is as follows. Since the boundedness of the genus and gonality  is a birational invariant, we can run the MMP and replace $X$ with a birational model (with terminal singularities) such that either $K_X$ is nef or $X$ has a Mori fibre space structure $X/B$. In the  former case any birational self-map is a pseudo-automorphism \cite[Lemma~3.4]{Hanamura1987} and so the genus and gonality are bounded in this case. If  $X/B$ is a Mori fibre space, then any birational map $X\dashrightarrow X$ is a composition of Sarkisov links (see Sect.~\ref{Sect:Sarkisov}).
If a link involves a Mori fibre space $X_i/B_i$ which is (generically) a conic bundle, then we apply an  explicit construction of Sect.~\ref{Sect:conic-bundle} to get unboundedness (and thus obtain Theorem \ref{MainTheorem}\ref{ConicBundleDim3Unbounded}).  If a link involves Mori fibre spaces $X_i/B_i$ and
$X_{i+1}/B_{i+1}$, then we use the boundedness result for Fano threefolds \cite{KMMT-2000} (see also \cite{BirkarS}).
This result is also used to prove the assertions \ref{MainTheorem}\ref{NotConicBundleGonalityBounded}-\ref{NotConicBundleNotDp3GenusBounded} (see Lemma~\ref{Lemm:g0c0moregeneral}). The unboundeness of the genus for del Pezzo fibrations of degree $3$ (Theorem \ref{MainTheorem}\ref{DelPezzo3GenusUnbounded}) is obtained by finding $2$-sections of large genus and applying Bertini involutions associated to these curves, see Section~\ref{Sec:DP3} for the detailed construction.

\subsection{Acknowledgements}
This research was supported through the programme ``Research in Pairs'' by the Mathematisches Forschungsinstitut Oberwolfach in 2018. J\'er\'emy Blanc acknowledges support by the Swiss National Science
Foundation Grant ``Birational transformations of threefolds'' 200020\_178807.
Ivan Cheltsov and Yuri Prokhorov were partially supported by the Royal Society grant No. IES$\backslash$R1$\backslash$180205, and the Russian Academic Excellence Project 5-100.
Alexander Duncan was partially supported by National Security Agency grant H98230-16-1-0309.
We thank Serge Cantat, St\'ephane Lamy and Egor Yasinsky for interesting discussions during the preparation of this text.
We also thank the anonymous referee for providing thoughtful comments that resulted in changes to the revised version of the paper.

\section{The case of conic bundles}\label{Sect:conic-bundle}
Every conic bundle is square birational equivalent to a conic bundle that can be seen as a conic in a (Zariski locally trivial) $\p^2$-bundle. Using a rational section of the $\p^2$-bundle, one can do the following construction:

\begin{construction}
Let $\pi\colon Q\to B$ be a conic bundle over an irreducible normal variety $B$ and let $Q_\eta$ be its generic fibre.
Then $Q_\eta$ is a conic over the function field $\C(B)$. The anicanonical linear system $|-K_{Q_\eta}|$ defines
an embedding $Q_\eta \hookrightarrow \mathbb P^2_{\C(B)}$. Fix a $\C(B)$-point $s_\eta\in \mathbb P^2_{\C(B)}\setminus Q_\eta$.  The projection $\mathrm{pr}\colon Q_\eta\to \mathbb P^1_{\C(B)}$ from $s_\eta$ is a double cover. Let $\iota_\eta\colon Q_\eta \to Q_\eta$ be the corresponding Galois involution.  It induces a fibrewise  birational involution $\iota\colon Q\dashrightarrow Q$.

Suppose now that our conic bundle $\pi\colon Q\to B$ is embedded into a $\mathbb P^2$-bundle $\hat \pi\colon P\to B$ and suppose that we are given a section $s\colon B \to P$ whose image is not contained in $Q$. This section defines a point
$s_\eta\in \mathbb P^2_{\C(B)}\setminus Q_\eta$ and therefore defines an involution $\iota\colon Q\dashrightarrow Q$ as above.
\end{construction}
\begin{lemma}{\cite[Lemma 15]{BlancLamy15}}\label{Lemm:ConicBundleInvolution}
If, in the above notation, $\Gamma\subseteq B$ is an irreducible hypersurface that is not contained in the discriminant locus of $\pi$ and such that $s(\Gamma)\subseteq Q$, the hypersurface $V=\pi^{-1}(\Gamma)$ of $Q$ is contracted by $\iota$ onto the codimension $2$ subset $s(\Gamma)$.
\end{lemma}

\begin{corollary}\label{Coro:SectionContracts}
Let $\pi\colon Q\to B$ be a conic bundle over an irreducible normal variety $B$, given by the restriction of a $\p^2$-bundle $\hat{\pi}\colon P\to B$. Let $\Gamma\subseteq B$ be an irreducible hypersurface such that the restriction of $\pi$ gives a trivial conic bundle $V=\pi^{-1}(\Gamma)\to \Gamma$. Then, there exists an involution \[\iota\in\Bir(Q/B)=\{\varphi\in \Bir(Q)\mid \pi\varphi=\pi\}\] that contracts the hypersurface $V$ onto the image of a rational section of $\Gamma\dasharrow V$.
\end{corollary}
\begin{proof}
Since the restriction of $\pi$ gives a trivial conic bundle $V=\pi^{-1}(\Gamma)\to \Gamma$, there is a rational section $s_\Gamma\colon \Gamma\dasharrow V\subseteq Q\subseteq P$. We then extend this section to a rational section $s\colon B\dasharrow P$ whose image is not contained in $Q$. This can be done locally, on an open subset where $P\to B$ is a trivial $\p^2$-bundle. Lemma~\ref{Lemm:ConicBundleInvolution} provides an involution $\iota\in \Bir(Q/B)$ that contracts $V$ onto the image of $s_\Gamma$.
\end{proof}

We can now give the proof of Theorem~\ref{MainTheorem}\ref{ConicBundleDim3Unbounded} (and of the small generalisation to higher dimensions mentioned in the introduction):
\begin{proposition}\label{Prop:ConicBundlesUnbounded}
Let $B$ be a projective variety of dimension $\ge 2$, let $\pi\colon Q\to B$ be a conic bundle and let us assume that either $\pi$ is trivial \textup(admits a rational section\textup) or that $\dim(B)=2$. Then, the covering gonality \textup(and the genus if $\dim(Q)=3)$ of elements of $\Bir(Q/B)$ is unbounded.
\end{proposition}
\begin{proof}We can assume that $\pi$ is the restriction of a $\p^2$-bundle $\hat{\pi}\colon P\to B$.

Let $\Gamma\subseteq B$ be an irreducible hypersurface which is not contained in the discriminant locus of $\pi$. Then the restriction of $\pi$ gives a conic bundle $\pi_\Gamma\colon V=\pi^{-1}(\Gamma)\to \Gamma$. If $\pi$ is a trivial conic bundle, then so is $\pi_\Gamma$. If $\dim(B)=2$, then $\Gamma$ is a curve, and $\pi$ is again a trivial conic bundle by Tsen's Theorem \cite[Corollary 6.6.2 p.~232]{Kollar}. In both cases, we can apply Corollary~\ref{Coro:SectionContracts} to find an element of $\Bir(Q/B)$ that contracts the hypersurface $V\subseteq Q$, birational to $\p^1\times \Gamma$.

This can be done for any irreducible hypersurface of $B$ not contained in the discriminant locus.
Thus, the covering gonality of elements of $\Bir(Q/B)$ are unbounded (to see this, simply embed $B$ in a projective space and take a general hypersurface of large degree).  The same argument applies to the genus when $\dim(Q)=3$.
\end{proof}

We recall the following classical result:
\begin{lemma}\label{Lemm:Dp4ConicBundle}
Let $B$ be a projective curve. A del Pezzo fibration $X/B$ of degree $\ge 4$ is birational to a conic bundle $X'/B'$.
\end{lemma}
\begin{proof}
The generic fibre   of $X/B$ is a del Pezzo surface $F$ of degree $d\ge 4$ over the function field $\C(B)$. Applying MMP over $B$, we can assume that the generic fibre satisfies $\operatorname{rk} \operatorname{Pic}(F)=1$ or has a structure of conic bundle. In the latter case, the proof is over, so we assume that  $\operatorname{rk} \operatorname{Pic}(F)=1$. It is sufficient to show that $F$ birationally has a conic bundle structure $F\to C$, where $C$ is a curve defined over $\C(B)$; this is for instance the case if $F$ is rational. As $B$ is a curve, the field $\C(B)$ has the $\mathrm{C}_1$ property,
so $F$ has a rational $\C(B)$-point $x\in F$ (see \cite[Theorem IV.6.8, page 233]{Kollar}).

If $d=9$, the existence of $x$ implies that $F$ is isomorphic to $\p^2$.  If $d=8$, the fact that  $\operatorname{rk} \operatorname{Pic}(F)=1$ implies that $F$ is isomorphic to a smooth quadric in $\p^3$, and the projection from $x$ gives a birational map to $\p^2$. We cannot have $d=7$, as the unique $(-1)$-curve of $F_{\overline{\C(B)}}$ would be defined over $\C(B)$, contradicting the assumption $\operatorname{rk} \operatorname{Pic}(F)=1$.

It remains to study the cases where $d\in \{4,5,6\}$.  Since $\operatorname{rk} \operatorname{Pic}(F)=1$, the $\C(B)$-rational point $x\in F$ does not lie on a $(-1)$-curve. Therefore, by blowing-up $x\in F$ we obtain a del Pezzo surface $Y$ over $\C(B)$ of degree $d-1$ with $\operatorname{rk} \operatorname{Pic}(Y)=2$. Thus on $Y$ there exists a Mori contraction $Y\to Y'$ which is different from $Y\to F$. The type of $Y\to F$ can be computed explicitely (see \cite[Theorem~2.6]{IskovskikhFactorisation}):
 If $d=5$ (resp. $d=6$), then  $Y\to Y'$ is a birational contraction to $\mathbb P^2_{\C(B)}$, (resp. to a quadric in $\mathbb P^3_{\C(B)}$ having a rational point), so $F$ is again rational. If $d=4$, then $Y'\simeq \mathbb P^1_{\C(B)}$ and $Y\to Y'$ is a conic bundle.
This proves our lemma.
\end{proof}

\section{Reminders on the Sarkisov program}
\label{Sect:Sarkisov}
\begin{definition} \label{def:MorFibreSpace}
A variety $X$ with a surjective morphism $\eta\colon X\to B$ is a \emph{Mori fibre space} if the following conditions hold:
\begin{enumerate}
\item \label{def:fib_eta} $\eta$ has connected fibres, $B$ is normal, $\dim X > \dim B \ge 0$ and the relative Picard rank $\rho(X/B)=\rho(X)-\rho(B)$ is equal to $1$;
\item \label{def:fib_sing} $X$ is $\QQ$-factorial with at most terminal singularities;
\item \label{def:fib_big} The anticanonical divisor $-K_X$ is $\eta$-ample.
\end{enumerate}
The Mori fibre space is denoted by $X/B$.

An isomorphism of fibre spaces $X/B\iso X'/B'$ is an isomorphism $\varphi\colon X\to X'$ that sits in a  commutative diagram
\[
 \xymatrix{
 X\ar[d]\ar[rr]^{\varphi}&&X'\ar[d]
 \\
B \ar[rr]^{\psi}&&B'
 }
\]
where $\psi\colon B\iso B'$ is an isomorphism.
\end{definition}
\begin{remark} In the case that we study, namely when $\dim(X)=3$, we obtain three possible cases for a Mori fibre space $X/B$:

\begin{enumerate}
\item
If $\dim(B)=0$, then $X$ is a Fano variety of Picard rank $1$;
\item
If $\dim(B)=1$, then $X$ is a del Pezzo fibration over the curve $B$;
\item
If $\dim(B)=2$, then $X$ is a conic bundle over the surface $B$.
\end{enumerate}
\end{remark}

\begin{definition}
A Sarkisov link $\chi\colon X_1\rat X_2$ between two Mori fibre spaces $X_1/B_1$ and $X_2/B_2$ is a birational map which fits into one of the following commutative diagrams.
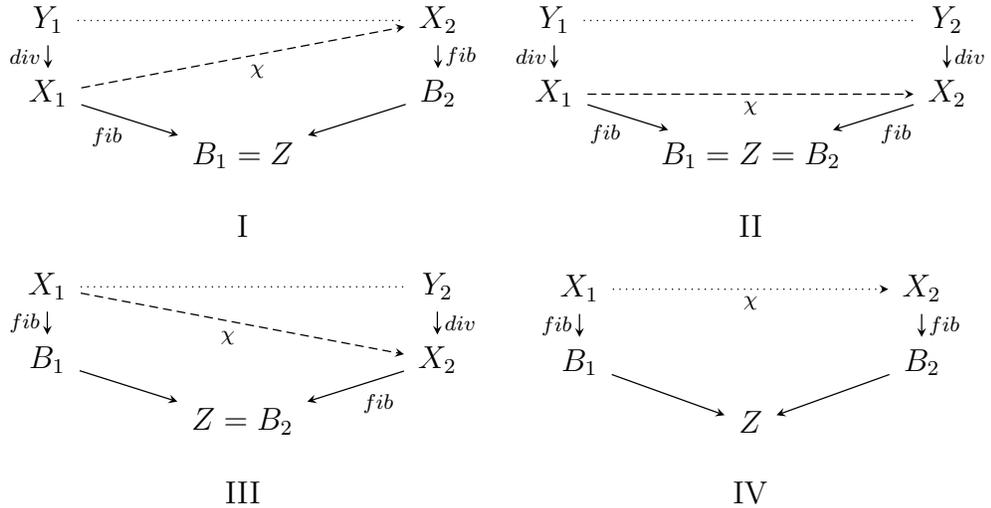
\[
{
\def\arraystretch{2.2}
\begin{array}{cc}
\begin{tikzcd}[ampersand replacement=\&,column sep=1.3cm,row sep=0.16cm]
Y_1\ar[dd,"div",swap]  \ar[rr,dotted,-] \&\& X_2 \ar[dd,"fib"] \\ \\
X_1 \ar[uurr,"\chi",dashed,swap] \ar[dr,"fib",swap] \&  \& B_2 \ar[dl] \\
\& B_1 = Z \&
\end{tikzcd}
&
\begin{tikzcd}[ampersand replacement=\&,column sep=.8cm,row sep=0.16cm]
Y_1
\ar[dd,"div",swap]  \ar[rr,dotted,-] \&\& Y_2\ar[dd,"div"] \\ \\
X_1 \ar[rr,"\chi",dashed,swap] \ar[dr,"fib",swap] \&  \& X_2 \ar[dl,"fib"] \\
\& B_1 = Z = B_2 \&
\end{tikzcd}
\\
\mathrm{I} & \mathrm{II}
\\
\begin{tikzcd}[ampersand replacement=\&,column sep=1.3cm,row sep=0.16cm]
X_1 \ar[ddrr,"\chi",dashed,swap] \ar[dd,"fib",swap]  \ar[rr,dotted,-] \&\& Y_2\ar[dd,"div"] \\ \\
B_1 \ar[dr] \& \& X_2 \ar[dl,"fib"] \\
\& Z = B_2 \&
\end{tikzcd}
&
\begin{tikzcd}[ampersand replacement=\&,column sep=1.5cm,row sep=0.16cm]
X_1 \ar[rr,"\chi",dotted,swap] \ar[dd,"fib",swap]  \&\& X_2 \ar[dd,"fib"] \\ \\
B_1 \ar[dr] \& \& B_2 \ar[dl] \\
\& Z \&
\end{tikzcd}
\\
\mathrm{III} & \mathrm{IV}
\end{array}
}
\]
Here the dotted arrows are pseudo-isomorphisms (isomorphisms outside of codimension $\ge 2$ subsets) given by a sequence of log-flips, the plain arrows are surjective morphisms of relative Picard rank $1$, with fibres not equivalent via $\chi$ in the cases of types $\mathrm{II}$ and $\mathrm{IV}$, the arrows written {\it ``div''} are divisorial contractions, the arrows written {\it ``fib''} are Mori fiber spaces,
and the variety $Z$ is normal with at worst Kawamata log terminal singularities.  We say that the \emph{base of the Sarkisov link} is the variety $Z$ (which is dominated by, but not necessarily equal to, the bases $B_1$ and $B_2$ of the two Mori fibre spaces), and that the above diagram is the \emph{Sarkisov diagram associated to~$\chi$}.
\end{definition}

The notion of Sarkisov links is important, because of the following result.
\begin{theorem}\label{TheoremSarkisov}
Every birational map between Mori fibre spaces decomposes into a composition of Sarkisov links and isomorphisms of Mori fibre spaces.
\end{theorem} 
In dimension $2$, this is essentially due to Castelunovo \cite{Cas}, although not stated directly in these terms. The case of dimension $3$ was done for the first time in \cite[Theorem 3.7]{CortiSarkisov}. The proof in any dimension is available in \cite[Theorem~1.1]{HMcK}.

\begin{remark}
\label{Rem:WeakFanoTop}
In fact, it follows from the definition that there are strong constraints on the sequence of anti-flips, flops and flips (that is, about the sign of the
intersection of the exceptional curves against the canonical divisor).  Precisely, as explained in \cite[Remark 3.10]{BLZ}, the top (dotted) row of a Sarkisov diagram has the following form:
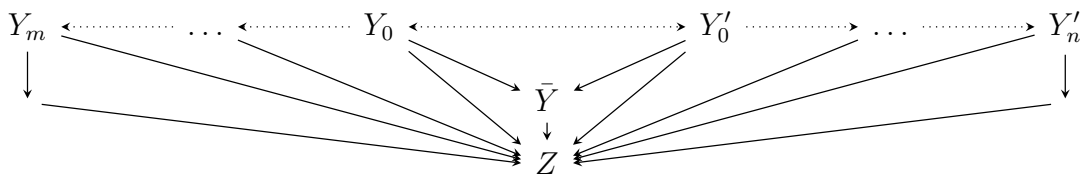
\[
\begin{tikzcd}[column sep=1.5cm,row sep=0.25cm]
Y_m \ar[ddrrr]\ar[d] \ar[r,dotted,<-]  & \dots\ar[ddrr] \ar[r,dotted,<-] & Y_0 \ar[ddr] \ar[rr,dotted,<->]  \ar[rd,->] &{}& Y'_0 \ar[ddl] \ar[r,dotted,->] \ar[ld,->]& \dots \ar[ddll]\ar[r,dotted,->] & Y'_n \ar[d] \ar[ddlll]\\
{} \ar[drrr] &{}&{}&\bar{Y}\ar[d]&{}&{}& {}  \ar[dlll] \\
{}  &{}&{}&Z&{}&{}& {}
\end{tikzcd}
\]
where $Y_0 \ps Y'_0$ is a flop over $Z$ (or an isomorphism), $m,n \ge 0$, and each $Y_i \ps Y_{i+1}$, $Y'_i \ps Y'_{i+1}$ is a flip over $Z$ (or an isomorphism).
Indeed, one can decompose the pseudo-isomorphism into a sequence of log-flips and for $Y = Y_i$ or $Y=Y_i'$, a general contracted curve $C$ of the fibration $Y/Z$ satisfies $K_Y \cdot C < 0$, hence at least one of the two extremal rays of the cone $\NE(Y/Z)$ is strictly negative against $K_Y$. In particular, both $Y_0/Z$ and $Y_0'/Z$ are relatively weak Fano (or Fano if the flop is an isomorphism) over $Z$.

If $Y_0 \ps Y'_0$ is an isomorphism, we choose $\bar Y$ to be isomorphic to both; if $Y_0 \ps Y'_0$ is a flop and not an isomorphism, the map is naturally associated to a variety $\bar Y$, that is a Fano with terminal (but not $\QQ$-factorial)  singularities such that $\operatorname{rk} \operatorname{Cl}(\bar Y/Z)=2$. It is called \emph{the central model} in \cite{Shokurov-Choi-2011}. Two contractions $Y_0\to \bar Y$ and $Y_0'\to \bar Y$ are small $\QQ$-factorialisations of $\bar Y$. Hence the whole diagram is uniquely determined by $\bar Y/Z$.
\end{remark}

\begin{remark}\label{Rem:Links3Dnoconicbundle}
In the sequel, we will mostly work with varieties of dimension $3$ not birational to conic bundles, as the case of conic bundles have already been treated in Section~\ref{Sect:conic-bundle}. The Mori fibre spaces will be then either Fano of rank $1$ or del Pezzo fibrations of degree $\le 3$ (see Lemma~\ref{Lemm:Dp4ConicBundle}). As a Fano is rationally connected, a del Pezzo fibration over a base not equal to $\p^1$ is not birational to a Fano variety. All the Sarkisov links that we can have between Mori fibre spaces not birational to conic bundles are then as follows:
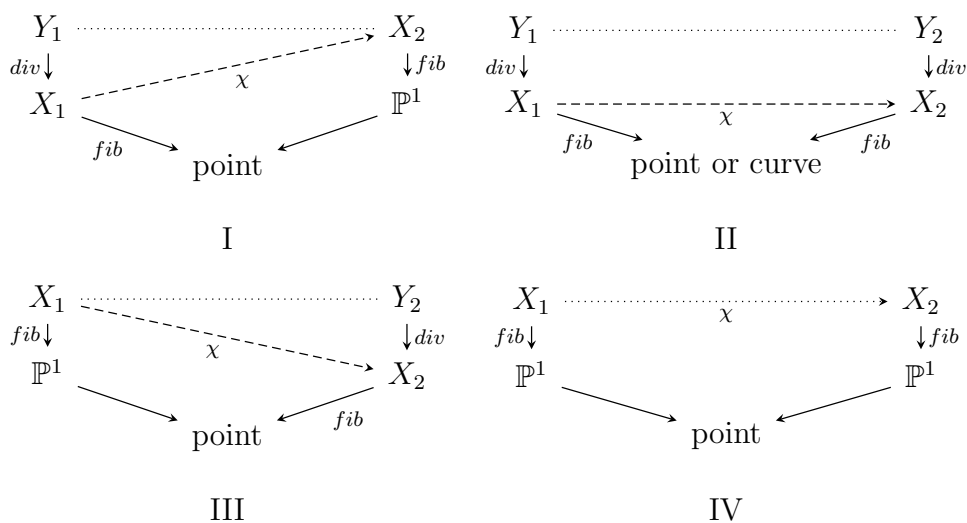
\[
{
\def\arraystretch{2.2}
\begin{array}{cc}
\begin{tikzcd}[ampersand replacement=\&,column sep=1.3cm,row sep=0.16cm]
Y_1\ar[dd,"div",swap]  \ar[rr,dotted,-] \&\& X_2 \ar[dd,"fib"] \\ \\
X_1 \ar[uurr,"\chi",dashed,swap] \ar[dr,"fib",swap] \&  \& \p^1 \ar[dl] \\
\& \text{point} \&
\end{tikzcd}
&
\begin{tikzcd}[ampersand replacement=\&,column sep=.8cm,row sep=0.16cm]
Y_1
\ar[dd,"div",swap]  \ar[rr,dotted,-] \&\& Y_2\ar[dd,"div"] \\ \\
X_1 \ar[rr,"\chi",dashed,swap] \ar[dr,"fib",swap] \&  \& X_2 \ar[dl,"fib"] \\
\&\text{point or curve}\&
\end{tikzcd}
\\
\mathrm{I} & \mathrm{II}
\\
\begin{tikzcd}[ampersand replacement=\&,column sep=1.3cm,row sep=0.16cm]
X_1 \ar[ddrr,"\chi",dashed,swap] \ar[dd,"fib",swap]  \ar[rr,dotted,-] \&\& Y_2\ar[dd,"div"] \\ \\
\p^1 \ar[dr] \& \& X_2 \ar[dl,"fib"] \\
\& \text{point} \&
\end{tikzcd}
&
\begin{tikzcd}[ampersand replacement=\&,column sep=1.5cm,row sep=0.16cm]
X_1 \ar[rr,"\chi",dotted,swap] \ar[dd,"fib",swap]  \&\& X_2 \ar[dd,"fib"] \\ \\
\p^1 \ar[dr] \& \& \p^1 \ar[dl] \\
\& \text{point} \&
\end{tikzcd}
\\
\mathrm{III} & \mathrm{IV}
\end{array}
}
\]
\end{remark}

\begin{lemma}
\label{Lemm:TwotypesII}
Let us consider a Sarkisov link of type $\mathrm{II}$ between three-dimensional Mori fibre spaces $X_1/B$ and $X_2/B$
over a base $B$ of dimension $1$.
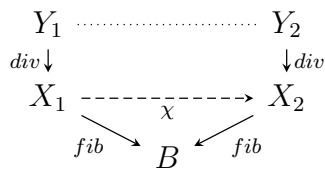
\[\begin{tikzcd}[ampersand replacement=\&,column sep=.8cm,row sep=0.16cm]
Y_1
\ar[dd,"div",swap]  \ar[rr,dotted,-] \&\& Y_2\ar[dd,"div"] \\ \\
X_1 \ar[rr,"\chi",dashed,swap] \ar[dr,"fib",swap] \&  \& X_2 \ar[dl,"fib"] \\
\&B\&
\end{tikzcd}\]
Denoting by $E_i\subset Y_i$ the exceptional divisor of $Y_i/X_i$ and by $e_i\subset X_i$ its image, one of the following case holds:
\begin{enumerate}
\item\label{typeIIfibre}
$\chi$ induces an isomorphism between the generic fibres of  $X_1/B$ and $X_2/B$, and $e_i$ is contained in a fibre of $X_i/B$ for $i=1,2$.
\item\label{typeIIcurvesurj}
$\chi$ induces a birational map between the generic fibres of  $X_1/B$ and $X_2/B$ which is not an isomorphism and $e_i$ is a curve of $X_i$ such that $e_i/B$ is a finite morphism of degree $r_i\in \{1,\ldots,8\}$, for $i=1,2$.
\end{enumerate}
Moreover, in case~\ref{typeIIcurvesurj}, if one of the degree $d_i$ of the del Pezzo fibration $X_i/B$ is $\le 3$, then $d_1=d_2$ and $r_1=r_2$, and  $(d_i,r_i)\in \{(3,2),(3,1),(2,1)\}$, and the generic fibres of  $X_1/B$ and $X_2/B$ are isomorphic.
In particular, $e_i$ is birational to $B$ if $d_i= 2$.
\end{lemma}

\begin{proof}
The image $e_i$ is a curve or a point, so is either $(a)$ contained in a fibre of $X_i/B$, or $(b)$ maps surjectively to $B$ via a finite morphism of degree $d_i\ge 1$. Case $(a)$ happens if and only if the generic fibres of $X_i/B$ and $Y_i/B$ are isomorphic.
As the generic fibres of $X_i/B$ are del Pezzo surfaces of rank $1$,  for $i=1,2$, case $(a)$ happens for $i=1$ if and only if it happens for $i=2$. This provides the dichotomy \ref{typeIIfibre}-\ref{typeIIcurvesurj} above.

In case~\ref{typeIIcurvesurj}, we look at the birational map between the generic fibres $X_i/B$ which are del Pezzo surfaces of degree $1$.
The classification of such maps, given in \cite[Theorem~2.6]{IskovskikhFactorisation},
implies that $r_i\le 8$ for $i=1,2$ and that if $d_i\le 3$ then $d_1=d_2$, $r_1=r_2$,  and  $(d_i,r_i)\in \{(3,2),(3,1),(2,1)\}$,
and the generic fibres of $X_1/B$ and $X_2/B$ are isomorphic.
\end{proof}

In Case~\ref{typeIIcurvesurj} in Lemma~\ref{Lemm:TwotypesII},
the case of del Pezzo fibrations of degree $3$ is the most interesting one,
as the degree $\le 2$ only gives curves $e_i$ of bounded genus, and the case of degree $\ge 4$ is covered by Lemma~\ref{Lemm:Dp4ConicBundle}.

\begin{remark}
\label{remark:fiberwise}
Let $X/B$ be a Mori fibre space such that $\dim(B)\in \{0,1\}$.
It may happen that no Sarkisov link starts from $X$.
If $\dim(B)=0$, this means (almost by definition) that $X$ is a birationally super-rigid Fano threefold.
Many examples of such Fano threefolds can be found in \cite{CheltsovPark}.
Similarly, if $\dim(B)=1$ and no Sarkisov link starts from $X$,
then $X/B$ is a del Pezzo fibration of degree $1$ that is birationally rigid over $B$ (see \cite[Definition~1.3]{CortiLMS}).
Vice versa, if $X/B$ is a del Pezzo fibration of degree $1$ that is birationally rigid over $B$,
then the only Sarkisov links that can start from $X$ are described in Case~\ref{typeIIfibre} of Lemma~\ref{Lemm:TwotypesII}.
For some (birationally rigid over the base) del Pezzo fibrations such links do not exist (see \cite{Krylov}).
However, in general they may exist and are not well understood (see \cite{Park1,Park2}).
\end{remark}

\section{Bounding the gonality and genus of curves}
We first state a consequence of the boundedness of weak-Fano terminal
varieties.
 The next lemma applies to Sarkisov links involving a Fano threefold of rank~$1$ (when one of the $B_i$ has dimension $0$) and to Sarkisov links of type $\mathrm{IV}$ between del Pezzo fibrations (when $\dim(B_1)=\dim(B_2)=1$ and $\dim(Z)=0$).

\begin{lemma}\label{Lemm:BoundDim0} There are integers $g_0,c_0\ge 1$
such that, for each  Sarkisov link
 $\chi\colon X_1\rat X_2$  between two Mori fibre spaces $X_1/B_1$ and $X_2/B_2$  such that $\dim(X_1)=\dim(X_2)=3$ over a base $Z$ of dimension $0$,  the following hold:

\begin{enumerate}
\item\label{divisorial}
Each divisorial contraction involved in the Sarkisov diagram of $\chi$ contracts a divisor birational to $\p^1\times \Gamma$ where $\gon(\Gamma)\le c_0$ and $g(\Gamma)\le g_0$.
\item\label{Fibres}
For each $i\in\{1,2\}$ such that $\dim(B_i)=1$, each fibre of $X_i/B_i$ is birational to $\p^1\times \Gamma$ with $\gon(\Gamma)\le c_0$ and   $g(\Gamma)\le g_0$.
\end{enumerate}
\end{lemma}
\begin{proof}
As in Remark~\ref{Rem:WeakFanoTop}, we consider the variety $Y_0$ which is pseudo-isomorphic
to the top varieties in the Sarkisov diagram and which is a weak-Fano variety of
rank 2, since the base $Z$ of the Sarkisov link is of dimension $0$. 
In particular, we see that $Y_0$ is $\mathrm{FT}^{\mathrm{t}}$ (see Definition~\ref{def:FT}),
so that all top varieties in the Sarkisov diagram are also $\mathrm{FT}^{\mathrm{t}}$ by Corollary~\ref{cor:FT-flips}.
Hence, the assertion~\eqref{divisorial} follows from Corollary~\ref{rcorollary:bound-genus},
and the assertion~\eqref{Fibres} follows from Corollary~\ref{corollary:FT-del-Pezzp-fibers}.
\end{proof}

The following result is a direct consequence of Lemma~\ref{Lemm:BoundDim0}.

\begin{corollary}
\label{corollary:solid}
There are integers $g_0,c_0\ge 1$ such that for each Mori fibre space $X/B$ where $\dim(X)=3$, $\dim(B)=0$,
not birational to a conic bundle or a del Pezzo fibration,
the genus and the gonality of elements of $\Bir(X)$ are bounded by $g_0$ and $c_0$ respectively.
\end{corollary}

\begin{proof}
Every element of $\Bir(X)$ decomposes into a product of Sarkisov links and isomorphisms of Mori fibre spaces (Theorem~\ref{TheoremSarkisov}). The base of each of these Sarkisov links has dimension $0$, as $X$ is not birational to a conic bundle or a del Pezzo fibration. Hence the genus and the gonality of each Sarkisov link are bounded by $g_0$ and $c_0$ respectively by Lemma~\ref{Lemm:BoundDim0}\ref{divisorial}. This provides the result.
\end{proof}

In fact, Corollary~\ref{corollary:solid} can be easily extended to higher dimension, and concerns then the \emph{solid Fano varieties}, defined as below (see  \cite[Definition~1.4]{AhmadinezhadOkada}):
\begin{definition}
A Fano variety $X$ being a Mori fibre space over a point is said to be  $($birationally$)$ \emph{solid} if it is not birational to any Mori fibre space over a positive dimensional base.
\end{definition}

\begin{corollary}
\label{corollary:solid-any-dimensions}
For every integer $n\ge 3$, there is an integer $c_n\ge 1$ $($that only depends on $n)$
such that for each solid Fano variety $X$ of dimension $n$,
the covering gonality of elements of $\Bir(X)$ are bounded by $c_n$.
\end{corollary}
\begin{proof}
The proof is similar as the one of Lemma~\ref{Lemm:BoundDim0} and  Corollary~\ref{corollary:solid}:
As $X$ is a solid Fano, every element of $\Bir(X)$ decomposes into a finite number of Sarkisov links between Mori fibre spaces over a base of dimension $0$.  As in Remark~\ref{Rem:WeakFanoTop}, we consider the variety $Y_0$ which is pseudo-isomorphic
to the top varieties in the Sarkisov diagram and which is a weak-Fano variety of
rank 2, since the base $Z$ of the Sarkisov link is of dimension $0$. The result then follows from Corollary~\ref{rcorollary:bound-genus}. 
\end{proof}

We can now extend Corollary~\ref{corollary:solid} to a more general situation:

\begin{lemma}
\label{Lemm:g0c0moregeneral}
There are integers $g_0,c_0\ge 1$ such that the following holds:

Suppose $X/B$ is a Mori fibre space with $\dim(X)=3$ and there exists a Sarkisov link over a base of dimension $0$ involving $X/B$.
Then, for every element $\varphi \in Bir(X)$ having a decomposition into Sarkisov links that involves no conic bundles, the gonality of $\varphi$ is bounded by $c_0$. Moreover, if we can choose a decomposition of $\varphi$ such that no del Pezzo fibration of degree $\ge 3$ arises, the genus of $\varphi$ is bounded by $g_0$.
\end{lemma}
\begin{proof}
We choose integers $g_0$ and $c_0$ from Lemma~\ref{Lemm:BoundDim0}, and assume that $c_0\ge 8$.

If $\dim(B)=0$, then $X$ is a Fano variety, and thus rationally connected \cite{Zhang}. If $\dim(B)\ge 1$, then $\dim(B)=1$ as $X$ is not birational to a conic bundle. By assumption, $X/B$ is involved in a Sarkisov link over a base of dimension $0$. The variety $X$ is then birational to a Fano variety (the variety $\bar Y$ of Remark~\ref{Rem:WeakFanoTop}), so is again rationally connected.  Moreover, Lemma~\ref{Lemm:BoundDim0}\ref{Fibres} implies that each fibre of $X/B$ is birational to $C\times \p^1$ for some curve of genus and gonality bounded by $g_0$ and $c_0$.

We take an element $\varphi\in \Bir(X)$  that we decompose, using Theorem~\ref{TheoremSarkisov}), as \[\varphi=\theta_r\circ \chi_r\circ \cdots \circ \theta_1\circ \chi_1\circ \theta_0,\] where each $\chi_i$ is a Sarkisov link between $X_{i-1}'/B_{i-1}'$ and $X_{i}/B_{i}$, each $\theta_i$ is an isomorphism of Mori fibre spaces
\[\begin{tikzcd}[ampersand replacement=\&,column sep=.8cm,row sep=0.16cm]
X_i
\ar[dd]  \ar[rr,"\theta_i"] \&\& X_i'\ar[dd] \\ \\
B_i \ar[rr,"\simeq",swap] \&  \& B_i'
\end{tikzcd}\]
 where $X_0/B_0=X_r'/B_r'=X/B$. By assumption, we may choose one such decomposition such that none of the $X_i/B_i$ (or $X_i'/B_i'$) is a conic bundle, which means that $\dim(B_i)\in \{0,1\}$ for each $i\in \{0,\ldots,r\}$.

By Remark~\ref{Rem:Links3Dnoconicbundle} and Lemma~\ref{Lemm:TwotypesII}, we obtain three different types of Sarkisov links $\chi_i$:

$a)$ \emph{Sarkisov links $\chi_i$ with a base of dimension $0$.}

$b)$ \emph{Sarkisov links $\chi_i$ of type $\mathrm{II}$ over a curve $B_{i-1}=B_i$ inducing no isomorphism between the generic fibres of $X_{i-1}'/B_{i-1}$ and $X_i/B_i$.}

$c)$ \emph{Sarkisov links $\chi_i$ of type $\mathrm{II}$ over a curve $B_{i-1}=B_i$ inducing an isomorphism between the generic fibres of $X_{i-1}'/B_{i-1}$ and $X_i/B_i$.}

The genus and the gonality of the Sarkisov link in case $a)$ are bounded by $g_0$ and $c_0$ respectively (Lemma~\ref{Lemm:BoundDim0}\ref{divisorial}).

In Case $b)-c)$, the Sarkisov link is between two del Pezzo fibrations $X_{i-1}'/B_{i-1}'$ and  $X_i/B_i$ over a curve $B_{i-1}'=B_i$. As $X_{i-1}$ and $X_i$ are rationally connected (because they are birational to $X$), we obtain $B_{i-1}'=B_i \simeq \p^1$.

Case $b)$ corresponds to Case~\ref{typeIIcurvesurj} of Lemma~\ref{Lemm:TwotypesII}). The two divisorial contractions contract divisors onto  curves of $X_{i-1}'/B_{i-1}'$ and $X_{i}/B_{i}$ of gonality at most $8$, which are moreover rational (and thus of genus bounded by $g_0$) if the degree of the del Pezzo fibrations is $\le 2$ (Lemma~\ref{Lemm:TwotypesII}).

The remaining case is Case $c)$ (Case~\ref{typeIIfibre} of Lemma~\ref{Lemm:TwotypesII}). In this case, the surfaces $F_{i-1}\subseteq X_{i-1}'$ and $F_{i}\subseteq X_{i}$ contracted by $\chi_i$ and $\chi_i^{-1}$ correspond to fibres of the del Pezzo fibrations $X_{i-1}'/B_{i-1}'$ and $X_{i}/B_{i}$. If the fibre is general, it is rational and we are done, but it could be that it is birational to $\p^1\times C$ for some non-rational curve $C$, a priori of large genus / gonality even if we do not know if such a situation is really possible (see Question~\ref{Ques:Unboundedfibres}). To overcome this issue, we denote by $1\le j\le i\le k\le r$ the smallest integer $j\in \{1,\dots,i\}$ and the biggest integer $k\in \{i,\dots,r\}$  such that $\chi_j,\ldots,\chi_k$ are links of type $\mathrm{II}$, and obtain that \[\nu=\theta_k\circ \chi_k\circ \theta_{k-1}\circ\chi_{k-1}\circ \cdots \circ \chi_{j+1}\circ \theta_j\circ \chi_j\circ \theta_j\] is a birational map between del Pezzo fibrations $X_{j-1}/B_i$ and $X_{k}'/B_{k}'$ which fits in a commutative diagram
\[\begin{tikzcd}[ampersand replacement=\&,column sep=.8cm,row sep=0.16cm]
X_{j-1}
\ar[dd]  \ar[rr,"\nu",dashed] \&\& X_k'\ar[dd] \\ \\
B_i\ar[rr,"\simeq",swap] \&  \& B_k'\simeq \p^1.
\end{tikzcd}\]
Every surface contracted by $\nu$ is either a fibre, or birational to a surface contracted by a Sarkisov link in $\chi_j,\ldots,\chi_k$ of type $b)$.
We now prove that the fibres contracted by $\nu$ are birational to $C\times \p^1$ for some curve $C$ of genus and  gonality bounded by $g_0$ and $c_0$ respectively. If $j=1$, this is because $X_0/B_0=X/B$ and we already observed at the beginning that each fibre of $X/B$ had this property. If~$j>1$, then the Sarkisov link $\chi_{j-1}$ is not of type $\mathrm{II}$, so is over a base of dimension~$0$, i.e. is in Case a). Hence, each fibre of $X_{j-1}/B_{j-1}$ is birational to $C\times \p^1$ for some curve $C$ of genus $\le g_0$ and gonality $\le c_0$ (Lemma~\ref{Lemm:BoundDim0}\ref{Fibres}).

Hence, even if $\chi_i$ can a priori be of arbitrary large genus or gonality, the gonality of $\nu$ is bounded by $c_0$, and the genus is bounded by $g_0$ if no del Pezzo fibration of degree $3$ appears in the decomposition.

As $\varphi$ decomposes into links of type $a)$ and maps having the same form as $\nu$ (compositions of Sarkisov links of type $\mathrm{II}$),
the gonality of $\nu$ is bounded by $c_0$, and the genus is bounded by $g_0$ if no del Pezzo fibration of degree $3$ appears in the decomposition.\end{proof}

The following question is naturally raised by the proof of Lemma~\ref{Lemm:g0c0moregeneral}.
\begin{question}\label{Ques:Unboundedfibres}
Is there an integer $g\ge 1$ such that for each Mori fibre space $X\to B$ which is a del Pezzo fibration, each fibre is birational to $C\times \p^1$ for some curve $C$ of genus (respectively gonality) $\le g$?
\end{question}

In the case of birational maps between del Pezzo fibrations, we do not have an absolute bound (which would follow from a positive answer to Question~\ref{Ques:Unboundedfibres}), but we can easily obtain the following result on birational maps involving links over a base of dimension $1$. This is for instance the case for all elements of $\Bir(X)$ if $X/B$ is  a Mori fibre space not birational to a conic bundle with $X$ not rationally connected.

\begin{lemma}\label{Lemm:OnlyDelPezzo}
Let $X/B$ be a Mori fibre space such that $\dim(X)=3$ and \mbox{$\dim(B)=1$} (a del Pezzo fibration over a curve).
There are integers $c,g\ge 0$ (depending on $X/B$) such that the following holds:

For each birational map $\varphi\in\Bir(X)$ that decomposes into Sarkisov links of type $\mathrm{II}$,
each over a base of dimension $1$, the gonality of $\varphi$ is bounded by $c$.
Moreover, the genus of $\varphi$ is bounded by $g$ if no del Pezzo fibration of degree $\ge 3$ occurs in the decomposition.
\end{lemma}

\begin{proof}
As each Sarkisov link occuring in the decomposition of $\varphi$ is of type $\mathrm{II}$, the base of the Sarkisov link is isomorphic to $B$, and so are all bases of the Mori fibre spaces involved.

The map $\varphi$ is a square birational map, i.e.~sends a general fibre of $X/B$ onto a general fibre.
There are finitely many fibres of $X/B$ that are not rational, so we only need to bound the gonality and the genus of the curves $C$
such that $C\times \p^1$ is birational to a surface contracted by $\varphi$ that is not a fibre.
Such a surface is contracted by a Sarkisov link $\chi$ between del Pezzo fibrations $X_{i-1}/B_{i-1}$ and $X_{i}/B_{i}$,
which is not an isomorphism between the generic fibres.
We proceed as in the proof of Lemma~\ref{Lemm:g0c0moregeneral}: The two divisorial contractions contract divisors
onto  multisections of $X_{i-1}/B_{i-1}$ and $X_{i}/B_{i}$.
As the variety $Y_0$ in the middle (see Remark~\ref{Rem:WeakFanoTop}) is a weak del Pezzo fibration over the base,
we can only blow-up curves with gonality at most $8$.
Moreover, the curves are rational if $X_{i-1}/B_{i-1}$ and $X_{i}/B_{i}$ are del Pezzo fibrations of degree $\le 2$.
This achieves the result.
\end{proof}

We can now apply Lemma~\ref{Lemm:g0c0moregeneral} and obtain the following result, which gives the proof of parts~\ref{NotConicBundleGonalityBounded} and~\ref{NotConicBundleNotDp3GenusBounded} of Theorem~\ref{MainTheorem}. Note that here the bound depends on $X$ and is not an absolute bound as in Lemma~\ref{Lemm:g0c0moregeneral}. This is of course needed, as one can consider the blow-up of any threefold along a curve of arbitrary large genus.

\begin{proposition}
\label{Prop:BoundedThreefolds}
Let $X$ be a projective threefold not birational to a conic bundle. Then, then the following hold:
\begin{enumerate}
\item\label{Prop:BoundedThreefolds1}
The gonality of the elements of $\Bir(X)$ is bounded.
\item\label{Prop:BoundedThreefolds2}
If $X$ is not birational to a del Pezzo fibration of degree $3$, the genus of the elements of $\Bir(X)$ is also bounded.
\end{enumerate}
\end{proposition}

\begin{proof}
Using the minimal model program (MMP), we see that $X$ is birational to a $\QQ$-factorial variety $X'$ with at worst terminal singularities, which is either a Mori fibre space $X'/B'$ or where $K_{X'}$ is nef. We may replace $X$ with $X'$, as this does not change the boundedness (but can a priori change the bound).

If $K_{X}$ is nef, then the genus and the gonality of elements of $\Bir(X)$ are bounded, as no element of $\Bir(X)$  contracts any hypersurface.

We can then assume that $X/B$ is a Mori fibre space. We have $\dim(B)\in \{0,1\}$ as $X$ is not birational to a conic bundle.
If no Sarkisov link starts from $X$, the result is trivially true, since in this case Theorem~\ref{TheoremSarkisov} gives
\[\Bir(X)=\Aut(X).\]
Note that such threefolds do exists (see Remark~\ref{remark:fiberwise}).
To complete the proof, we may assume that $\Bir(X)\ne\Aut(X)$,
so that, in particular, there is a Sarkisov link that starts from $X$.
If $X/B$ is involved in a Sarkisov link over a base of dimension~$0$, the result follows from Lemma~\ref{Lemm:g0c0moregeneral}.
So we may assume that the base of every Sarkisov link between Mori fibre spaces birational to $X$ has dimension $1$.
Then $\dim(B)=1$ and every Sarkisov link involved in any decomposition of any element $\varphi\in\Bir(X)$ is of type $\mathrm{II}$ between two del Pezzo fibrations.
In particular, the del Pezzo fibration $X/B$ is birationally rigid over $B$ (see \cite[Definition~1.3]{CortiLMS} and Remark~\ref{remark:fiberwise}).
Now the result follows from Lemma~\ref{Lemm:OnlyDelPezzo}.
\end{proof}

\section{Del Pezzo fibrations of degree \texorpdfstring{$3$}{3}}\label{Sec:DP3}
We recall the following result, proven in \cite{CortiAnnals} (see also \cite{KollarEle} for generalisations):
\begin{proposition}[{\cite[Theorem 1.10]{CortiAnnals}}]
\label{Prop:Corti}
Let $B$ be a smooth curve and let $X\to B$ be a del Pezzo fibration of degree $3$ \textup(cubic surface fibration\textup). Then, there exists a birational map
\[
 \xymatrix{
 X\ar[dr]\ar@{-->}[rr]^{\varphi}&&X'\ar[dl]
 \\
 &B&
 }
\]
such that  $X'/B$ is another del Pezzo fibration of degree $3$ having the following properties:
\begin{enumerate}
 \item\label{Cond2}
 $X'$ is a projective threefold with terminal singularities of index $1$;
 \item\label{Cond3}
 every fibre of $X'/B$ is reduced and irreducible, and is a Gorenstein del Pezzo surface;
 \item\label{Cond4}
 the anti-canonical system $-K_{X'}$ is relatively very ample and defines an embedding in a $\p^3$-bundle over $B$.
 \end{enumerate}
\end{proposition}
We will also need the following lemma:

\begin{lemma}\label{lemma:blowup}
Let $0\in V\subset \mathbb C^4$ be an isolated  cDV singularity and let $C\subset V$ be a smooth curve that contains $0$.
Let $\sigma:\hat V\to V$ be the blowup of $C$. Then $\hat V$ is normal.
\end{lemma}

\begin{proof}
Let $t$ be a local parameter on $C$.
Take analytic coordinates $x_1,\dots, x_4$ in $\mathbb C^4$ so that $x_1=t$.
Then $C$ is the $x_1$-axis, given by $x_{2}=x_{3}=x_{4}=0$. As $V$ contains $C$, it is given by the equation
\[
\phi=x_2\phi_2+x_3\phi_3+x_4\phi_4=0,
\]
where the functions  $\phi_i=\phi_i(x_1,\dots,x_4)$ vanish at the origin.
The origin being a cDV singularity, at least one of $\phi_i$'s contains a linear term and
at least one of the $\phi_i$'s contains a term $x_1^k$ for some $k>0$ (because
$V$ is smooth at a general point of~$C$).
By changing coordinates $x_2,x_3,x_4$ linearly, we may assume that
$\phi_4$ contains a non-zero linear term $\ell(x_1,x_4)$. Note that the fibre $\sigma^{-1}(0)\subset \hat V$ is a plane $\mathbb P^2$ and $\hat V$ is a hypersurface in a nonsingular fourfold. By Serre's criterion it is sufficient to show that  $\hat V$ is smooth at some point of $\sigma^{-1}(0)$.
Consider the affine chart $U_4:=\{x_4' \neq 0\}$.
Then $\sigma^{-1}(U_4)$
is given by
\[
\{x_2'\phi'_2+x_3'\phi'_3+\phi'_4=0\}\subset \mathbb C^4_{x_1',\dots,x_4'},
\]
where $\phi'_i=\phi_i(x_1',x_2'x_4',x_3'x_4',x_4')$. The fibre  $\sigma^{-1}(0)\cap U_4$
in this chart is given by $x_1'=x_4'=0$. Then
\[
\operatorname{mult}_0 (x_2'\phi'_2) \ge 2,\quad \operatorname{mult}_0 (x_3'\phi'_3) \ge 2,
\]
and the linear part $\ell(x_1,x_4)$ of $\phi_4'$ is nontrivial. Therefore,
$\hat V$ is smooth at the origin of   $\sigma^{-1}(U_4)$.
\end{proof}
\begin{proposition}\label{Prop:DP3GenusNotBounded}
Let $B$ be a smooth curve and let $X\to B$ be a del Pezzo fibration of degree $3$. Then, the genus of elements of $\Bir(X/B)$ is unbounded \textup(even if the gonality can be bounded\textup).
\end{proposition}
\begin{proof}
We can assume that $B$ is projective, and apply Proposition~\ref{Prop:Corti} to reduce to the case
where $X,B$ satisfy the conditions \ref{Cond2}-\ref{Cond4} of this proposition.

We then take a curve $C\subset X$ which is a section of $X/B$ (this exists as
the field $\C(B)$ has the $C_1$ property, by \cite[Theorem IV.6.8,
page 233]{Kollar}). We can moreover assume that a general point of $C$ is not
contained in any of the $27$ lines of the corresponding fibre (\cite[Theorem
IV.6.10, page 234]{Kollar}). We view $X$ as a closed hypersurface of a $\p^3$-bundle $P\to B$
(Condition~\ref{Cond4} of Proposition~\ref{Prop:Corti}). We consider the blow-up $\hat{\eta}\colon \hat{P}\to P$ of $P$ along $C$, denote by $\hat{X}\subset \hat{P}$ the strict transform of $X$, and denote by $\eta\colon \hat{X}\to X$ the restriction of $\hat{\eta}$, that is the blow-up of $C$.
As $\hat X$ is a hypersurface of $\hat{P}$,
 the canonical class $K_{\hat X}$
is a well-defined Cartier divisor.
By Lemma \ref{lemma:blowup}, $\hat X$ is normal.
Consider, for some large positive integer $n$, the divisor
\[
D=-K_{\hat{X}}+n\hat{F},
\]
where $\hat{F}$
is a general fibre of $\hat{X}\to B$.

We first observe that $D$ is base-point-free (for $n$ big enough). To see this,
we view $X$ as a closed hypersurface of a $\p^3$-bundle $P\to B$
(Condition~\ref{Cond4} of Proposition~\ref{Prop:Corti}), and then view $\hat{X}$
as a closed hypersurface of the blow-up $\hat{P}$ of $P$ along $C$. By the adjunction formula, the divisor
$D$ is the restriction of a divisor $D_P$ on $\hat{P}$ which is equal, on
each fibre of $\hat{P}\to B$, to a strict transform of a hyperplane of $\p^3$
through the point blown-up. Taking $n$ big enough, we obtain that $D_P$ is
without base-points, and this implies that $D$ is base-point free.

Take two general elements $D_1,D_2$ of the linear system of $D$ and consider the curve $\hat{Q}=D_1\cap D_2$. Observe that $\hat{Q}$ intersects a general fibre $\hat{F}$ of $\hat{X}\to B$ at two points. Indeed, for $i=1,2$, the intersection of $D_i$ with a general fibre $\hat{F}$ is the strict transform of a hyperplane section of the cubic surface $F$, fibre of $X\to B$, passing through the point $C\cap F$ blown-up by $\eta$. The two hyperplane sections intersect $F$ into $3$ points, so $2$ outside of $C\cap F$.

We can then associate to $Q=\eta(\hat{Q})$ the birational involution $\varphi\in \Bir(X/B)$ which performs a Bertini involution on a general fibre $F$ of $X/B$, associated to the two points $Q\cap F$: the fibre $F$ is a smooth cubic surface in $\p^3$ and the blow-up of the two points $Q\cap F$ is a del Pezzo surface of degree $1$ on which there is the unique Bertini involution associated to the double covering (see for instance \cite[Page 613]{IskovskikhFactorisation}). Hence, the involution $\varphi$ lifts to a birational involution of the blow-up $Y\to X$ of $Q$, which is an isomorphism on a general fibre. There is then a surface birational to $Q\times \p^1$ contracted by $\varphi$, which corresponds to the union of two curves contracted onto the two points in each fibre.

It remains to see that $\hat{Q}$ is smooth and irreducible, and that the genus $g(\hat{Q})$ is strictly increasing as $n$ grows.
 This will give the result, as $Q$ is then birational to $\hat{Q}$.
Since the linear system $\lvert D\rvert$ is base point free and ample, it is not composed with a pencil.
Since $\hat X$ is normal, it is smooth in codimension one so
$\hat{Q}=D_1\cap D_2$ is a smooth and irreducible curve, contained in
the smooth locus of $\hat{X}$. To see this, we apply Bertini's
theorem twice for the smoothness and use the fact that the support of an ample divisor is connected.

By adjunction formula we get, for
$i=1,2$,
\[\begin{array}{rclcl}
K_{D_1}&=&(K_{\hat X}+D_1)|_{D_1}&=&n\cdot (\hat{F}|_{D_1}),\\
K_{Q}&=&(K_{D_1}+Q)|_{D_1}&=&(n\hat{F}|_{D_1}+D_2|_{D_1})|_{Q},\\
\end{array}\]
which gives
\[\begin{array}{rcl}
\deg(-K_Q)&=& (n\hat{F}+D_2)\cdot D_1\cdot D_2\\
&=& (n\hat{F}+D)\cdot D^2\\
&=&(-K_{\hat{X}}+2n\hat{F})\cdot (-K_{\hat{X}}+n\hat{F})^2\\
&=&4n\hat{F}\cdot (K_{\hat{X}})^2-(K_{\hat{X}})^3\\
&=&4n-(K_{\hat{X}})^3
\end{array}\]
This shows that the genus of the curve $Q$, birational to $\hat{Q}$, strictly increases as $n$ grows.
\end{proof}

The proof of Theorem~\ref{MainTheorem} is now finished:

\begin{proof}[Proof of Theorem~$\ref{MainTheorem}$]
Part~\ref{ConicBundleDim3Unbounded} is given by Proposition~\ref{Prop:ConicBundlesUnbounded}.

Part~\ref{DelPezzo3GenusUnbounded} is given by Proposition~\ref{Prop:DP3GenusNotBounded}.

Parts~\ref{NotConicBundleGonalityBounded} and~\ref{NotConicBundleNotDp3GenusBounded} are respectively parts \ref{Prop:BoundedThreefolds1} and \ref{Prop:BoundedThreefolds2} of Proposition~\ref{Prop:BoundedThreefolds}.
\end{proof}

\appendix\renewcommand{\thesubsection}{\Alph{section}.\arabic{subsection}}
\section{FT varieties}

In this appendix, we present some results from \cite{P-Sh:JAG},
\cite{BCHM}, and \cite{BirkarS}, with some mild changes of notation
and presentation. We simply recall them.
The notation $\mathrm{FT}^{\mathrm{t}}$ is a variation of the
usual notation for ``Fano type,'' where we add a {\it t} to indicate a
terminality condition.

\begin{definition}[cf. {\cite[\S~2]{P-Sh:JAG}}]
\label{def:FT}
Let $X$ be a normal projective variety. We say that $X$ is $\mathrm{FT}^{\mathrm{t}}$ (Fano type) if there exists a $\mathbb{Q}$-divisor
$\Delta=\sum b_i\Delta_i$, with $b_i>0$ for each $i$,  such that the following conditions are satisfied
\begin{enumerate}
\item\label{def:FT1}
the pair $(X,\Delta)$ is terminal,
\item\label{def:FT2}
$K_X+\Delta\equiv 0$,
\item\label{def:FT3}
the components $\Delta_i$ generate the group  $\Cl(X)\otimes \mathbb Q$,
\item\label{def:FT4}
each component $\Delta_i$ is a movable divisor
$($i.e.~the linear system $\lvert \Delta_i\rvert$ associated to
$\Delta_i$ is without fixed components$)$.\end{enumerate}
\end{definition}

\begin{lemma}$ $
\label{lemma:Theta}
\begin{enumerate}
 \item
 \label{lemma:Theta1}
A variety  $X$ is $\mathrm{FT}^{\mathrm{t}}$  if and only if there exists  an effective $\mathbb{Q}$-divisor $\Theta$
such that  the pair $(X,\Theta)$ is terminal, $-(K_X+\Theta)$ is ample,
and the components of $\Theta$ are movable divisors that generate  $\Cl(X)\otimes \mathbb Q$.
\item
\label{lemma:Theta2}
A $\mathbb{Q}$-factorial variety  $X$ is $\mathrm{FT}^{\mathrm{t}}$  if and only if there exists
an effective $\mathbb{Q}$-divisor $\Theta$
such that  the pair $(X,\Theta)$ is terminal and $-(K_X+\Theta)$ is ample.
\end{enumerate}
\end{lemma}

\begin{proof}
\ref{lemma:Theta1}
Let $X$ be an $\mathrm{FT}^{\mathrm{t}}$ variety and let $A$ be an ample divisor. By our assumption
\ref{def:FT}\ref{def:FT3}
we have $A\equiv \sum c_i \Delta_i$
for some $c_i\in \mathbb Q$. Take $\Theta:=\Delta-\epsilon\sum c_i \Delta_i$
for $0<\epsilon \ll 1$.  Conversely, assume that there exists  $\Theta$ as in \ref{lemma:Theta1}.
For $n\gg 0$ the linear system  $|-n(K_X+\Theta)|$ is base point free.
Take a general member $A\in |-n(K_X+\Theta)|$ and put $\Delta=\Theta+\frac 1n A$.

\ref{lemma:Theta2}
If $X$ is $\mathrm{FT}^{\mathrm{t}}$, the  existence of $\Theta$ follows from \ref{lemma:Theta1}. Conversely, let  $\Theta$ be  a boundary
such that  the pair $(X,\Theta)$ is terminal and $-(K_X+\Theta)$ is ample. As $X$ is $\mathbb{Q}$-factorial, we may
take very ample divisors  $A_1,\dots, A_m$ generating $\Cl(X)\otimes \mathbb Q$.
Then we can apply \ref{lemma:Theta1} to the boundary $\Theta'=\Theta+\epsilon \sum A_i$   for  $0<\epsilon \ll 1$.
\end{proof}

\begin{corollary}
\label{cor:FT}
If $X$ is an $\mathrm{FT}^{\mathrm{t}}$ variety, then the numerical and $\mathbb Q$-linear equivalences
of $\mathbb Q$-Cartier divisors on $X$ coincide.
\end{corollary}

\begin{proof}
By Lemma~\ref{lemma:Theta}\ref{lemma:Theta1}, there exists an effective $\mathbb{Q}$-divisor $\Theta$ such that  the pair $(X,\Theta)$ is terminal and $-(K_X+\Theta)$ is ample. This implies that $(X,\Theta)$ is a log Fano variety, so  $\mathbb Q$-linear equivalences
of $\mathbb Q$-Cartier divisors on $X$ coincide  \cite[Proposition~2.1.2]{IP99}.
\end{proof}

\begin{lemma}
\label{lemma:FT-bir-small}
Let $\varphi\colon X\to X'$ be a small birational contraction.
Then  $X$ is \FT if and only if  $X'$ is.
\end{lemma}
\begin{proof}
For all $\mathbb{Q}$-divisors $\Delta$ on $X$ and  $\Delta'$ on $X'$
such that $\Delta'=\varphi_*(\Delta)$, the conditions
\ref{def:FT2}-\ref{def:FT3}-\ref{def:FT4} of \ref{def:FT} are fullfilled
for $\Delta$ if and only if they are fullfilled for $\Delta'$. If
$K_X+\Delta$ and $K_{X'}+\Delta'$ are numerically trivial
$\mathbb{Q}$-Cartier divisors, \cite[Lemma~3.38]{Kollar-Mori-1988} implies that $(X,\Delta)$ is terminal if and only if $(X',\Delta')$ is terminal.

Assume that $X$ is \FT.
Let $\Delta$ be as in Definition~\ref{def:FT} and let $\Delta':=\varphi(\Delta)$.
By Corollary~\ref{cor:FT}, the divisor $K_{X'}+\Delta'$ is $\mathbb Q$-Cartier. As $(X,\Delta)$ is terminal, so is  $(X',\Delta')$.

Conversely, assume that $X'$ is \FT and  $\varphi$ is small.
As above, we have $K_{X}+\Delta=\varphi^*(K_{X'}+\Delta')$, where $(X',\Delta')$ is terminal.
Thus $(X,\Delta)$ is terminal as well.
\end{proof}

\begin{corollary}
\label{cor:FT-flips}
The \FT property is preserved under flips.
\end{corollary}
\begin{lemma}
\label{lemma:FT-bir-div}
Let $X$ be a  $\mathbb{Q}$-factorial \FT variety  $X$
and let $\varphi\colon X\to \bar X$ be a divisorial extremal contraction.
Then  $\bar X$ is \FT.
\end{lemma}
\begin{proof}
Let $\Delta=\sum b_i\Delta_i$ be as in Definition~\ref{def:FT}. By Lemma~\ref{lemma:Theta}\ref{lemma:Theta1}, there exists an effective $\mathbb{Q}$-divisor $\Theta$ such that $-(K_X+\Theta)$ is ample. We may then replace $\Delta$ with $\Theta+\frac 1n A$, where $A\in |-n(K_X+\Theta)|$
for sufficiently big and divisible $n$, and assume that $\Delta_1$ is ample. 

We write $\bar \Delta:=\varphi(\Delta)=\sum b_i\bar \Delta_i$ where
$\bar \Delta_i:=\varphi_*(\Delta_i)$.
Let $\bar A$ be an ample divisor on $\bar X$. We can write $\varphi^*\bar A\equiv \sum \beta_i \Delta_i$ for some $\beta_i\in \mathbb{Q}$.  As $b_i>0$ for each $i$, we can choose $\lambda,\epsilon>0$ small enough such that $b_i>\lambda \beta_i$, and $b_1>\lambda \beta_1+\epsilon$. Let
\[
\Delta_{\lambda,\epsilon}:= \Delta-\lambda \sum \beta_i \Delta_i - \epsilon \Delta_1,\qquad
\bar \Delta_{\lambda,\epsilon}:=\varphi_*\Delta_{\lambda,\epsilon}.
\]
Thus, $\Delta_{\lambda,\epsilon}$ and $\bar \Delta_{\lambda,\epsilon}$ are effective and
\[
-(K_X+\Delta_{\lambda,\epsilon})\equiv \lambda \varphi^*\bar A + \epsilon \Delta_1
\]
so $-(K_X+\Delta_{\lambda,\epsilon})$ is $\varphi$-ample.
For $\lambda \ll 1$ the pair $(X,\Delta_{\lambda,\epsilon})$ is
terminal. Since the ray contracted by $\varphi$ is
$(K_X+\Delta_{\lambda,\epsilon})$-negative,
and no components of the boundary $\Delta_{\lambda,\epsilon}$ are contracted by $\varphi$ since they are movable by assumption, the log pair $(\bar X,\bar \Delta_{\lambda,\epsilon})$ is terminal as well \cite[Corollary~3.43]{Kollar-Mori-1988}.
Furthermore, for $\epsilon\ll  \lambda$ the divisor
\[
-(K_{\bar X}+\bar \Delta_{\lambda,\epsilon})\equiv \lambda \bar A + \epsilon \bar \Delta_1
\]
is ample.
Thus $\bar X$ is \FT by Lemma~\ref{lemma:Theta}\ref{lemma:Theta2}.
\end{proof}

Note that, by the cone theorem, for an \FT variety $X$ the Mori cone $\overline{\mathrm{NE}}(X)$ is polyhedral.
In particular, the number of extremal rays on $X$ is finite. We moreover have the following stronger conditions on $X$:

\begin{corollary}[cf.~{\cite[Corollary 2.7]{P-Sh:JAG},\cite[Corollary~1.3.2]{BCHM}}]
\label{cor:MDS}
A $\mathbb Q$-factorial \FT variety is a Mori dream space.
In particular, on such a variety  one can run the $D$-MMP with respect to any divisor $D$.
\end{corollary}

\begin{proof}
Let $X$ be a $\mathbb Q$-factorial \FT variety. By Lemma~\ref{lemma:Theta}\ref{lemma:Theta1}, there exists an effective $\mathbb{Q}$-divisor $\Theta$ such that  the pair $(X,\Theta)$ is terminal and $-(K_X+\Theta)$ is ample. This implies that $X$ is a Mori dream space \cite[Corollary~1.3.2]{BCHM}.

Let $\Delta=\sum b_i\Delta_i$ be as in Definition~\ref{def:FT}.
For each divisor $D$, we can write $D\equiv \sum d_i\Delta_i$ (by  \ref{def:FT}\ref{def:FT3}).
Thus for $0<\epsilon \ll 1$ there is a divisor $\Delta_D:=\Delta-\epsilon \sum d_i\Delta_i$
 such that $(X,  \Delta_D)$ is terminal and $-(K_X+\Delta_D)\equiv D$. We can thus run a  $D$-MMP, as this is is the same as the $(K_X+\Delta_D)$-MMP.
\end{proof}

\begin{corollary}
\label{cor:FTnef}
 Let  $X$ be an \FT variety. Then there exists a birational map $\psi\colon X \dashrightarrow X'$
 such that $\psi$ is an isomorphism in codimension one and such that
$X'$ has only $\mathbb Q$-factorial terminal singularities and $-K_{X'}$ is nef and big $(X'$ is weak Fano$)$.
\end{corollary}

\begin{proof}
According to Lemma~\ref{lemma:FT-bir-small}, we may replace $X$ with its small $\mathbb Q$\-/factorialization. Run the $-K_{X}$-MMP.
Since $-K_X\equiv \Delta$, where $\Delta$ is effective, on each step the exceptional locus
is contained in the proper transform of $\Delta$.
By  \ref{def:FT}\ref{def:FT4} none of components of $\Delta$ are contracted.
Hence all the steps of the MMP are flips.
We end up with  a variety  $X'$,  which is \FT by Corollary~\ref{cor:FT-flips}, with nef anticanonical divisor. By Lemma~\ref{lemma:Theta}\ref{lemma:Theta1}, there exists an effective $\mathbb{Q}$-divisor $\Theta$ such that  $-(K_{X'}+\Theta)$ is ample. This implies that $-K_{X'}$ is  big.
\end{proof}

\begin{lemma}
Let $X \dashrightarrow X'$ be a birational map that is an
isomorphism in codimension $1$. If $X$ is \FT , then so is $X'$.
\end{lemma}

\begin{proof}Let $\Delta$ be as in Definition~\ref{def:FT}.
Let $X \xlongleftarrow{p}\tilde X\xlongrightarrow{q} X'$ be
a common log resolution. Write
\[
K_{\tilde X}+\tilde \Delta=p^*(K_X+\Delta )+E.
\]
where $\tilde \Delta$ is the proper transform of $\Delta$ and  $E$ is
an effective divisor whose support coincides with the exceptional locus.

Since $K_{\tilde X}+\tilde \Delta$ is numerically equivalent to $E$ over $X$ and $E>0$, we may run the $(K_{\tilde X}+\tilde \Delta)$-MMP over $X'$ and contract all the components of $E$. We end up with a terminal pair
$(\hat X,\hat \Delta)$ having a small contraction to $X'$, where $\hat \Delta$ is the proper transform of $\Delta$.
Moreover,  $K_{\hat X}+\hat \Delta\equiv 0$. Hence $\hat X$ is of \FT type.
The variety $X'$ is \FT by Lemma~\ref{lemma:FT-bir-div}.
\end{proof}

From Lemma~\ref{lemma:Theta} and  \cite[Theorem~1.1]{BirkarS} we have the following.

\begin{corollary}
\label{cor:bound}
Let $n\ge 2$ be an integer. The \FT varieties of dimension $n$ form a bounded family. This means that there exists a flat morphism $\mathfrak{h}\colon \mathfrak {X}\to \mathfrak {S}$
over a scheme $\mathfrak {S}$ of finite type such that each  \FT variety of dimension $n$ is isomorphic to a fiber of $\mathfrak{h}$.
\end{corollary}
\begin{proof}
For each  \FT  variety $X$ of dimension $n$, there exists an effective $\mathbb{Q}$-divisor $\Theta$ on $X$, such that  the pair $(X,\Theta)$ is terminal (Lemma~\ref{lemma:Theta}) and $-(K_{X}+\Theta)$ is ample.  Hence, we may apply \cite[Theorem~1.1]{BirkarS} to the pairs  $(X,\Theta)$ with $\epsilon=1$.
\end{proof}
In the sequel, when we say that a set of varieties is bounded, we always
mean, as above, that there is a flat family over  a scheme  of finite
type such that every variety in the set is isomorphic to a fiber.

\begin{remark}
\label{rem:boundK}
Let  $\{ X_\alpha\}$ be a set of $\mathbb Q$-factorial \FT varieties of dimension $n$
and let $\mathfrak{h}\colon \mathfrak {X}\to \mathfrak {S}$  be a flat family as in \textup{\ref{cor:bound}}. We identify the varieties $X_\alpha$ with fibers of $\mathfrak{h}$ and may assume that  the set $\{ \mathfrak{h}(X_\alpha)\}$ is dense in  $\mathfrak {S}$.
Then by \cite{Kawamata:def-can} the total space $\mathfrak {X}$ has only canonical singularities.
In particular,  $\mathfrak {X}$ is $\mathbb Q$-Gorenstein.
This implies that $(-K_{X_\alpha})^n$ is bounded. More precisely, the numbers $(-K_{X_\alpha})^n$ form a finite set that depends only on $n$.
\end{remark}

\begin{proposition}
\label{prop:bound-dcontr}
 Let  $\{ X_\alpha\}$ be a set of $\mathbb Q$-factorial \FT varieties of dimension $n$ and let
 $f_{\alpha,i}\colon X_{\alpha}\to Y_{\alpha,i}$ be a collection of divisorial Mori contractions.
For each $i$, we denote by $E_{\alpha,i}$ the exceptional divisor of $f_{\alpha,i}$.
Then the set $\{E_{\alpha,i}\}$ is bounded.
\end{proposition}

\begin{proof}
Let $\mathfrak{h}\colon \mathfrak {X}\to \mathfrak {S}$ be a flat family given 
by Corollary~\ref{cor:bound} and Remark~\ref{rem:boundK} so that 
$X_{\alpha}$ is the fiber over $\mathfrak {s}_{\alpha}\in \mathfrak {S}$.
Then for some ample divisor $\mathfrak{L}$ on $\mathfrak {S}$ the divisor $-K_{\mathfrak {X}}+\mathfrak{h}^*\mathfrak{L}$ 
is big on $\mathfrak {X}$.
Let $\mathfrak {A}$ be an ample divisor on $\mathfrak {X}$.
By Kodaira's lemma there exists an integer $m$ (that depends only on $n$) such that
$-m K_{\mathfrak {X}}+m\mathfrak{h}^*\mathfrak{L} -\mathfrak {A}$ is effective, i.e. $-m K_{\mathfrak {X}}+m\mathfrak{h}^*\mathfrak{L}\sim \mathfrak {A}+\mathfrak {D}$,
where $\mathfrak {D}$ is effective. According to \cite{Kawamata:length} the exceptional divisor $E_{\alpha,i}$ is covered by a family of curves $\{C^{\lambda}_{\alpha,i}\}$, such that $C^{\lambda}_{\alpha,i}\cdot E_{\alpha_i}<0$ in $X_\alpha$ and such that each $f(C^{\lambda}_{\alpha,i})$  is a point
and $-K_{X_{\alpha}}\cdot C^{\lambda}_{\alpha,i}\le 2n$. If  $\mathfrak {D}\cdot C^{\lambda}_{\alpha,i}<0$ for some $\alpha$ and $i$, then $E_{\alpha,i}$ is contained in $\mathfrak {D}$.
Hence, the degree of $E_{\alpha,i}$, with respect to the ample divisor $\mathfrak {A}$, can be bounded as follows
\[
E_{\alpha,i}\cdot (\mathfrak {A}|_{X_{\alpha}})^{n-1} \le \mathfrak {D}\cdot X_{\alpha}   \cdot \mathfrak {A}^{n-1} = (-m K_{\mathfrak {X}}+m\mathfrak{h}^*\mathfrak{L} -\mathfrak {A}) \cdot X_{\alpha}   \cdot \mathfrak {A}^{n-1}.
\]
Thus for such $\alpha$ and $i$ the set  $\{E_{\alpha,i}\}$ is bounded. 
From now on we may assume that $\mathfrak {D}\cdot C^{\lambda}_{\alpha,i}\ge 0$
and so $\mathfrak {A}\cdot C^{\lambda}_{\alpha,i}\le 2nm$. 
Hence $\{C^{\lambda}_{\alpha,i}\}$ is bounded, i.e.~the set of all curves $C^{\lambda}_{\alpha,i}$
is contained in a finite union of irreducible components
$\mathfrak {H}_{{\alpha,i}}$ 
of the relative Hilbert scheme of $\mathfrak {X}/ \mathfrak {S}$.
We may consider only one of them and put $\mathfrak {H}:=\mathfrak 
{H}_{{\alpha,i}}$. Let $\mathfrak{U}\to \mathfrak {H}$ be the universal family 
and let $\Psi\colon \mathfrak{U}\to \mathfrak{X}$ be the corresponding morphism.
Denote the composition by $\Phi=\mathfrak{h}\circ \Psi\colon \mathfrak{U}\to \mathfrak{S}$.
By construction any divisor  $E_{\alpha,i}$ 
coincides with
the image under 
$\Psi$ of a 
component 
of the fiber $\Phi^{-1}(\mathfrak {s}_{\alpha})$.
The $\Psi$-images of the components of the fibers of $\Phi$ lie in a finite number 
of components of the relative Hilbert scheme of $\mathfrak{X}$ over 
$\mathfrak{S}$. 
This means that  $\{E_{\alpha,i}\}$ is bounded.
\end{proof}

Proposition~\textup{\ref{prop:bound-dcontr}} implies that the birational invariants
of  $E_{\alpha,i}$ are bounded. In particular, we have the following.

\begin{corollary}
\label{rcorollary:bound-genus}
Let us take the notation of
Proposition~\textup{\ref{prop:bound-dcontr}}. There is an integer $m$,
depending only on $n$, such that each $E_{\alpha,i}$ is birationally
equivalent to $\mathbb P^1\times \Gamma_{\alpha,i}$ where the covering gonality of $\Gamma_{\alpha,i}$ satisfies $\covgon(\Gamma_{\alpha,i})\le m$. If moreover $n=3$, the set $\{g( \Gamma_{\alpha,i})\}$ of genera of the curves $ \Gamma_{\alpha,i}$ is a bounded set of integers.
\end{corollary}
\begin{proof}As explained in the introduction (see also
\cite{Kawamata:length}), every divisor $E_{\alpha_i}$ is birational to
$\mathbb P^1\times \Gamma_{\alpha,i}$ for some variety
$\Gamma_{\alpha,i}$  of dimension $n-2$. By
Proposition~\ref{prop:bound-dcontr}, the set $\{E_{\alpha,i}\}$ is
bounded. There is some integer $d$ such that each $E_{\alpha,i}$ is
isomorphic to a closed subvariety of a projective space of degree $\le
d$. A general hypersurface $H$ of $E_{\alpha,i}$ is of degree $\le d$
and admits a dominant rational map to $\Gamma_{\alpha,i}$. This gives
$d\ge \covgon(H)\ge \covgon(\Gamma_{\alpha,i})$. If moreover $n=3$,
then the curve $H$ is of bounded degree and thus of bounded genus,
which gives the result since $g(H)\ge g(\Gamma_{\alpha,i})$. 
\end{proof}

\begin{proposition}
\label{prop:bound-Fcontr}
 Let  $\{ X_\alpha\}$ be a set of $\mathbb Q$-factorial \FT varieties of 
dimension~$n$. Let
 $f_{\alpha,i}\colon X_{\alpha}\to Y_{\alpha,i}$ be a set of Mori 
fiber spaces and let $F_{\alpha,i,\beta}$ be a set of schematic fibers 
over smooth points $\beta\in Y_{\alpha,i}$. Suppose that $\dim 
(F_{\alpha,i,\beta})=\dim(X_{\alpha})-\dim(Y_{\alpha,i})$ for all  
$\alpha$, $i$, $\beta$. 
Then the set $\{F_{\alpha,i,\beta}\}$ is bounded.
\end{proposition}

\begin{proof}
Let $\mathfrak{h}\colon \mathfrak {X}\to \mathfrak {S}$ be a flat family given by Corollary~\ref{cor:bound} and Remark~\ref{rem:boundK}.
 Let $\mathfrak {A}$ be an ample divisor on $\mathfrak {X}$. As in the proof of Proposition~\ref{prop:bound-dcontr},  we may take an ample divisor $\mathfrak{L}$ on $\mathfrak {S}$ such that $-K_{\mathfrak {X}}+\mathfrak{h}^*\mathfrak{L}$ 
is big on $\mathfrak {X}$ and an integer $m$ (that depends only on $n$) such that
$-m K_{\mathfrak {X}}+m\mathfrak{h}^*\mathfrak{L} -\mathfrak {A}$ is effective, i.e. $-m K_{\mathfrak {X}}+m\mathfrak{h}^*\mathfrak{L}\sim \mathfrak {A}+\mathfrak {D}$,
where $\mathfrak {D}$ is effective. Denoting by $H_\alpha$ and $D_\alpha$ the restriction of $\mathfrak {A}$ and $\mathfrak{D}$ to $X_\alpha$, we obtain $-m K_X\sim H_\alpha+ D_\alpha$ where $D_\alpha$ is effective.

Note that $f_{\alpha,i}$ is flat in a neighborhood of $F_{\alpha,i, 
\beta}$.
Let $F_{\alpha,i}^{\mathrm{gen}}$ be a general fiber of $f_{\alpha,i}$ and let
$r_{\alpha,i}=\dim(F_{\alpha,i}^{\mathrm{gen}})=\dim(F_{\alpha,i, \beta})$.
By the adjunction formula $F_{\alpha,i}^{\mathrm{gen}}$ is a Fano
variety with terminal singularities. By \cite{BirkarS} (see
Remark~\ref{rem:boundK}) there is a constant $C=C(n)$ such that
\[
 (-K_{ F_{\alpha,i, \beta}})^{r_{\alpha,i}}=(-K_{ 
F^{\mathrm{gen}}_{\alpha,i}})^{r_{\alpha,i}} \le C.
\]
Since $f_{\alpha,i}\colon X_{\alpha}\to Y_{\alpha,i}$ is a Mori fiber space,
the restriction of  divisors $- K_{X_\alpha}$, $H_\alpha$, $D_\alpha$ to  $F_{\alpha,i}^{\mathrm{gen}}$
are proportional and nef. Hence the intersection numbers \[H_\alpha^i\cdot 
D_\alpha^{r_{\alpha,i}-i}\cdot F_{\alpha,i, \beta}=H_\alpha^i\cdot 
D_\alpha^{r_{\alpha,i}-i}\cdot F_{\alpha,i}^{\mathrm{gen}}\] are non-negative.
Thus we have
\[
C\ge (-K_{ F_{\alpha,i, 
\beta}})^{r_{\alpha,i}}=(-m_\alpha K_{X_\alpha})^{r_{\alpha,i}}\cdot 
F_{\alpha,i, \beta}=
 (H_\alpha+ D_\alpha)^{r_{\alpha,i}}\cdot F_{\alpha,i, \beta}\ge
  H_\alpha^{r_{\alpha,i}}\cdot F_{\alpha,i, \beta}.
\]
This shows that the degrees of $F_{\alpha,i, \beta}$
are bounded with respect to $\mathfrak {A}$.
\end{proof}

Note that all the  fibers  $F_{\alpha,i, \beta}$ are uniruled by 
\cite{Kawamata:length} and \cite[Corollary~1.5.1]{Kollar}.
Similar to Corollary~\textup{\ref{rcorollary:bound-genus}} in dimension $3$ we have  
the following. 

\begin{corollary}
\label{corollary:FT-del-Pezzp-fibers}
If, in the notation of Proposition~\textup{\ref{prop:bound-Fcontr}},
$n=3$ and $\dim(Y_{\alpha,i})=1$, then  every fiber
$F_{\alpha,i, \beta}$  is birationally equivalent to
the product of $\mathbb P^1$ and a curve $\Gamma_{\alpha,i}$
of bounded genus.
\end{corollary}
\begin{proof}
The base is  smooth since it is normal and 
one-dimensional.
Thus, Proposition~\ref{prop:bound-Fcontr} proves that the set   $\{F_{\alpha,i,\beta}\}$ of all fibers is bounded.
We then observe, as in Corollary~\textup{\ref{rcorollary:bound-genus}}, that this bounds the arithmetic 
genus of hyperplane sections of the  $F_{\alpha,i, \beta}$, and thus of the $\Gamma_{\alpha,i}$.
\end{proof}

\bibliography{genusproblem}
\bibliographystyle{alpha}
\end{document}